\documentclass{amsart}
\usepackage{graphicx}
\usepackage{amsmath,amssymb,amsfonts}
\newtheorem{theorem}{Theorem}[section]
\newtheorem{lemma}[theorem]{Lemma}
\newtheorem{coro}[theorem]{Corollary}

\theoremstyle{definition}
\newtheorem{definition}[theorem]{Definition}

\numberwithin{figure}{section}
\theoremstyle{remark}
\newtheorem{remark}[theorem]{Remark}

\numberwithin{equation}{section}

\def \A{\mathcal{A}}

\def \Q{\mathcal{Q}}
\def \la{\lambda}
\def \I{\mathcal{I}}

\begin{document}
\title[Least H-eigenvalue of hypergraphs]{The least H-eigenvalue of signless Laplacian of non-odd-bipartite hypergraphs}

\author[Y.-Z. Fan]{Yi-Zheng Fan$^*$}
\address{School of Mathematical Sciences, Anhui University, Hefei 230601, P. R. China}
\email{fanyz@ahu.edu.cn}
\thanks{$^*$The corresponding author. This work was supported by National Natural Science Foundation of China (Grant No. 11871073, 11771016).
}

\author[J.-C. Wang]{Jiang-Chao Wan}
\address{School of Mathematical Sciences, Anhui University, Hefei 230601, P. R. China}
\email{1500256209@qq.com}

\author[Y. Wang]{Yi Wang}
\address{School of Mathematical Sciences, Anhui University, Hefei 230601, P. R. China}
\email{wangy@ahu.edu.cn}


\begin{abstract}
Let $G$ be a connected non-odd-bipartite hypergraph with even uniformity.
The least H-eigenvalue of the signless Laplacian tensor of $G$ is simply called the least eigenvalue of $G$
  and the corresponding H-eigenvectors are called the first eigenvectors of $G$.
In this paper we give some numerical and structural properties about the first eigenvectors of $G$ which contains an odd-bipartite branch,
and investigate how the least eigenvalue of $G$ changes when an odd-bipartite branch attached at one vertex is relocated to another vertex.
We characterize the hypergraph(s) whose least eigenvalue attains the minimum
among a certain class of hypergraphs which contain a fixed non-odd-bipartite connected hypergraph.
Finally we present some upper bounds of the least eigenvalue and prove that zero is the least limit point of the least eigenvalues of connected non-odd-bipartite hypergraphs.
\end{abstract}



\maketitle

\section{Introduction}
Since Lim \cite{Lim} and Qi \cite{Qi2} independently introduced the eigenvalues of tensors or hypermatrices in 2005,
the spectral theory of tensors developed rapidly,
  especially the well-known Perron-Frobenius theorem of nonnegative matrices was generalized to nonnegative tensors \cite{CPZ, FGH, YY1, YY2, YY3}.
The signless Laplacian tensors $\Q(G)$ \cite{Qi}
were introduced to investigating the structure of hypergraphs, just like signless Laplacian matrices to simple graphs.
As $\Q(G)$ is nonnegative, by using Perron-Frobenius theorem, many results about its spectral radius are presented
\cite{HQX,KF,LSQ,LMZW,YQS}.

Let $G$ be a $k$-uniform connected hypergraph. 
Shao et al. \cite{SSW} prove that zero is an H-eigenvalue of $\Q(G)$ if and only if $k$ is even and $G$ is odd-bipartite.
Some other equivalent conditions are summarized in \cite{FY}.
Note that zero is an eigenvalue  of $\Q(G)$ if and only if $k$ is even and $G$ is odd-colorable \cite{FY}.
So, there exist odd-colorable but non-odd-bipartite hypergraphs \cite{FKT,Ni}, for which zero is an N-eigenvalue.
Hu and Qi \cite{Hu} discuss the H-eigenvectors of zero eigenvalue of $\Q(G)$ related to the odd-bipartitions of $G$,
and use N-eigenvectors of zero eigenvalue of $\Q(G)$ to discuss some kinds of partition of $G$,
where an eigenvector is called \emph{H}-(or \emph{N}-)\emph{eigenvector} if it can (or cannot) be scaled into a real vector.

Except the above work, the least H-eigenvalue of $\Q(G)$ receives little attention.
In this paper, we focus on the least H-eigenvalue of $\Q(G)$.
Qi \cite{Qi2} proved that each eigenvalue of $\Q(G)$ of a connected $k$-uniform hypergraph $G$ has a nonnegative real part by using Gershgorin disks, which implies that
the least H-eigenvalue of  $\Q(G)$ is at least zero, and is zero if and only if $k$ is even and $G$ is odd-bipartite.
If $k$ is even, then $\Q(G)$ are positive semi-definite \cite{Qi}, and its least H-eigenvalue is a solution of minimum problem over a real unit sphere; see Eq. (\ref{form2}).
So, throughout of this paper, when discussing the least H-eigenvalue of $\Q(G)$,
we always assume that \emph{$G$ is connected non-odd-bipartite with even uniformity $k$}.
For convenience, the least H-eigenvalue of $\Q(G)$ is simply called the \emph{least eigenvalue} of $G$
 and the corresponding H-eigenvectors are called the \emph{first  eigenvectors} of $G$.

In this paper we give some numerical and structural properties about the first  eigenvectors of $G$ which contains an odd-bipartite branch,
and investigate how the least eigenvalue of $G$ changes when an odd-bipartite branch
attached at one vertex is relocated to another vertex.
We characterize the hypergraph(s) whose least  eigenvalue attains the minimum
among a certain class of hypergraphs which contain a fixed non-odd-bipartite connected hypergraph.
Finally we present some upper bounds of the least  eigenvalue and prove that zero is the least limit point of the least  eigenvalues of connected non-odd-bipartite hypergraphs.
The perturbation result on the least  eigenvalue in this paper is a generalization of that on the least eigenvalue of the signless Laplacian matrix of a simple graph in \cite{WF}.

\section{Preliminaries}

\subsection{Eigenvalues of tensors}
A real {\it tensor} (also called \emph{hypermatrix}) $\A=(a_{i_{1} i_2 \ldots i_{k}})$ of order $k$ and dimension $n$ refers to a
  multi-dimensional array with entries $a_{i_{1}i_2\ldots i_{k}}\in \mathbb{R}$
  for all $i_{j}\in [n]:=\{1,2,\ldots,n\}$ and $j\in [k]$.
Obviously, if $k=2$, then $\A$ is a square matrix of dimension $n$.
The tensor $\A$ is called \emph{symmetric} if its entries are invariant under any permutation of their indices.

 Given a vector $x\in \mathbb{C}^{n}$, $\A x^{k} \in \mathbb{C}$ and $\A x^{k-1} \in \mathbb{C}^n$, which are defined as follows:
\begin{align*}
\A x^{k} & =\sum_{i_1,i_{2},\ldots,i_{k}\in [n]}a_{i_1i_{2}\ldots i_{k}}x_{i_1}x_{i_{2}}\cdots x_{i_k},\\
(\A x^{k-1})_i & =\sum_{i_{2},\ldots,i_{k}\in [n]}a_{ii_{2}\ldots i_{k}}x_{i_{2}}\cdots x_{i_k}, i \in [n].
\end{align*}

Let $\mathcal{I}=(i_{i_1i_2\ldots i_k})$ be the {\it identity tensor} of order $k$ and dimension $n$, that is, $i_{i_{1}i_2 \ldots i_{k}}=1$ if
   $i_{1}=i_2=\cdots=i_{k} \in [n]$ and $i_{i_{1}i_2 \ldots i_{k}}=0$ otherwise.

\begin{definition}[\cite{Lim,Qi2}]
Let $\A$ be a real tensor of order $k$ dimension $n$.
For some $\lambda \in \mathbb{C}$, if the polynomial system $(\lambda \mathcal{I}-\A)x=0$,
or equivalently $\A x^{k-1}=\lambda x^{[k-1]}$, has a solution $x \in \mathbb{C}^{n}\backslash \{0\}$,
then $\lambda $ is called an \emph{eigenvalue} of $\A$ and $x$ is an \emph{eigenvector} of $\A$ associated with $\lambda$,
where $x^{[k-1]}:=(x_1^{k-1}, x_2^{k-1},\ldots,x_n^{k-1})$.
\end{definition}

 In the above definition, $(\lambda,x)$ is called an \emph{eigenpair} of $\A$.
 If $x$ is a real eigenvector of $\A$, surely the corresponding eigenvalue $\lambda$ is real.
In this case, $\lambda$ is called an {\it H-eigenvalue} of $\A$.
Denote by $\lambda_{\min}(\A)$ the least H-eigenvalue of $\A$.

A real tensor $\A$ of even order $k$ is called \emph{positive semidefinite}  (or \emph{positive definite}) if for any $x\in \mathbb{R}^{n} \backslash \{0\}$,
$\A x^k \ge 0$ (or $\A x^k > 0$).

\begin{lemma}[\cite{Qi2}, Theorem 5] \label{semi}
Let $\A$ be a real symmetric tensor of order $k$ and dimension $n$, where $k$ is even.
Then the following results hold.

\begin{enumerate}

\item $\A$ always has H-eigenvalues, and $\A$ is positive definite (or positive semidefinite) if and only if its least H-eigenvalue is positive (or nonnegative).

\item $\lambda_{\min}(\A)=\min\{\A x^k: x\in \mathbb{R}^{n}, \|x\|_k=1\}$,
where $\|x\|_k=\left(\sum_{i=1}^{n} |x_i|^k\right)^{1 \over k}$.
Furthermore, $x$ is an optimal solution of the above optimization
if and only if it is an eigenvector of $\A$ associated with $\lambda_{\min}(\A)$.
\end{enumerate}
\end{lemma}

\subsection{Uniform hypergraphs}
A {\it hypergraph} $G=(V(G),E(G))$ is a pair consisting of a vertex set $V(G)=\{v_1,v_2,\ldots,v_n\}$ and an edge set $E(G)=\{e_{1},e_2,\ldots,e_m\}$,
where $e_j\subseteq V(G)$ for each $j\in[m]$.
If $|e_j|=k$ for all $j\in[m]$, then $G$ is called a {\it $k$-uniform} hypergraph.
The {\it degree} $d_G(v)$ or simply $d(v)$ of a vertex $v \in V(G)$ is defined as $d(v)=|\{e_{j}:v\in e_{j}\}|$.
The {\it order} of  $G $ is the cardinality of $V(G)$, denoted by $\nu(G)$, and its {\it size} is the cardinality of $E(G)$, denoted by $\varepsilon(G)$.
A {\it walk} in a $G$ is a sequence of alternate vertices and edges: $v_0e_1v_1e_2 \ldots e_lv_l$,
where ${v_i, v_{i+1}}\in e_i$ for $i = 0, 1, \ldots, l-1$.
A walk is called a {\it path} if all the vertices and edges appeared on the walk are distinct.
A hypergraph $G$ is called {\it connected} if any two vertices of $G$ are connected by a walk or path.

If a hypergraph is both connected and acyclic, it is called a \emph{hypertree}.
The {\it $k$-th power} of a simple graph $H$, denoted by $H^k$,
is obtained from $H$ by replacing each edge (a $2$-set) with a $k$-set by adding $(k-2)$ additional vertices \cite{HQS}.
The $k$-th power of a tree is called \emph{power hypertree}, which is surely a $k$-uniform hypertree.
In particular, the $k$-th power of a path $P_m$ (respectively, a star $S_m$)  (as a simple graph) with $m$ edges is called a \emph{hyperpath}  (respectively, \emph{hyperstar}), denote by $P_m^k$ (respectively, $S_m^k$).
In a $k$-th power hypertree $T$, an edge is called a \emph{pendent edge} of $T$ if it contains $k-1$ vertices of degree one,
which are called the \emph{pendent vertices} of $T$.

\begin{lemma}[\cite{Ber1}]\label{tree}
Let $G$ be a connected $k$-uniform hypergraph.
Then $G$ is a hypertree if and only if $\varepsilon(G)=\frac{\nu(G)-1}{k-1}.$
\end{lemma}

The odd-bipartite hypergraphs was introduced by Hu and Qi \cite{Hu}, which is considered as a generalization of the ordinary bipartite graphs.
The odd-bipartition is closely related to odd-traversal \cite{Ni}.

\begin{definition}[\cite{Hu}]
Let $k$ be even.
A $k$-uniform hypergraph $G=(V,E)$ is called \emph{odd-bipartite},
if there exists a bipartition $\{V_1,V_2\}$ of $V$ such that each edge of $G$ intersects $V_1$ (or $V_2$) in an odd number of vertices (such bipartition is called
an \emph{odd-bipartition} of $G$);
otherwise, $G$ is called \emph{non-odd-bipartite}.
\end{definition}

Let $G$ be a $k$-uniform hypergraph on $n$ vertices $v_1,v_2,\ldots,v_n$.
The {\it adjacency tensor} of $G$ \cite{CD} is defined as $\mathcal{A}(G)=(a_{i_{1}i_{2}\ldots i_{k}})$, an order $k$ dimensional $n$ tensor, where
\[a_{i_{1}i_{2}\ldots i_{k}}=\left\{
 \begin{array}{ll}
\frac{1}{(k-1)!}, &  \mbox{if~} \{v_{i_{1}},v_{i_{2}},\ldots,v_{i_{k}}\} \in E(G);\\
  0, & \mbox{otherwise}.
  \end{array}\right.
\]

 Let $\mathcal{D}(G)$ be a diagonal tensor of order $k$ and dimension $n$, where $d_{i\ldots i}=d(v_i)$ for $i \in [n]$.
The tensor $\Q(G)=\mathcal{D}(G)+\A(G)$ is called the {\it signless Laplacian tensor} of $G$ \cite{Qi}.
Observe that the adjacency (signless Laplacian) tensor of a hypergraph is symmetric.

Let $x=(x_1, x_2,\ldots,x_n) \in \mathbb{C}^n$.
Then $x$ can be considered as a function defined on the vertices of $G$,
 that is, each vertex $v_i$ is mapped to $x_i=:x_{v_ i}$.
 If $x$ is an eigenvector of $\Q(G)$, then it defines on $G$ naturally, i.e., $x_v$ is the entry of $x$ corresponding to $v$.
 If $G_0$ is a sub-hypergraph of $G$, denote by $x|_{G_0}$ the restriction of $x$ on the vertices of $G_0$, or a subvector of $x$ indexed by the vertices of $G_0$.

  Denote by $E_G(v)$, or simply $E(v)$, the set of edges of $G$ containing $v$.
 For a subset $U$ of $V(G)$, denote $x^U:=\Pi_{v \in U} x_u$, and $x_U^k:=\sum_{v \in U} x_u^k$.
 Then we have
 \begin{equation}\label{form}
  \Q(G)x^k = \sum_{e\in E(G)}(x_e^{k}+kx^e),
\end{equation}
 and for each $v \in V(G)$,
 $$(\Q(G)x^{k-1})_v=d(v)x_v^{k-1}+\sum_{e\in E(v)}x^{e\backslash \{v\}}.$$
So the eigenvector equation $\Q(G)x^{k-1}=\la x^{[k-1]}$ is equivalent to that for each $v \in V(G)$,
\begin{equation}\label{eigen}
(\la-d(v)) x_v^{k-1}=\sum_{e\in E(v)}x^{e\backslash \{v\}}.
\end{equation}
From Lemma \ref{semi}(2), if $k$ is even, then $\la_{\min}(G):=\lambda_{\min}(\Q(G))$ can be expressed as
\begin{equation}\label{form2}
\lambda_{\min}(G)=\min_{x\in \mathbb{R}^{n}, \|x\|_k=1}\sum_{e\in E(G)}(x_e^{k}+kx^e).
\end{equation}

Note that if $k$ is odd, the Eq. (\ref{form2}) does not hold.
The reason is as follows. If $G$ contains at least one edge, then by Perron-Frobenius theorem,
the spectral radius $\rho(\Q(G))$ of $\Q(G)$ is positive associated with a unit nonnegative eigenvector $x$.
Now $$\lambda_{\min}(G)\le \Q(G)(-x)^k=-\Q(G)x^k=-\rho(\Q(G))<0,$$ a contradiction as $\lambda_{\min}(G) \ge 0$ (see \cite[Theorem 3.1]{Qi}).

\begin{lemma}\label{xuv}
Let $G$ be a $k$-uniform hypergraph, and $(\lambda,x)$ be an eigenpair of $\Q(G)$.
If $E(u)=E(v)$ and $\la \ne d(u)$, then $x_u^k=x_v^k$.
\end{lemma}

\begin{proof}
Consider the eigenvector equation of $x$ at $u$ and $v$ respectively,
$$(\lambda-d(u))x_u^k=\sum_{e\in E(u)}x^e, \;(\lambda-d(v))x_v^k=\sum_{e\in E(v)}x^e.$$
As $E[u]=E[v]$, $d(u)=d(v)$ and $\sum_{e\in E(u)}x^e=\sum_{e\in E(v)}x^e$.
The result follows.
\end{proof}


\begin{lemma}[\cite{KF2}]\label{mindegree}
Let $G$ be a $k$-uniform hypergraph with the minimum degree $\delta(G) >0$, where $k$ is even.
 Then $\lambda_{\min}(G)<\delta(G)$.
\end{lemma}


\section{Properties of the first  eigenvectors}

Let $G_1$, $G_2$ be two vertex-disjoint hypergraphs, and let $v_1\in V(G_1), v_2\in V(G_2)$.
The \emph{coalescence} of $G_1$, $G_2$ with respect to $v_1, v_2$,
   denoted by $G_1(v_1)\diamond G_2(v_2)$, is obtained from $G_1$, $G_2$ by identifying $v_1$ with $v_2$ and forming a new vertex $u$.
The graph $G_1(v_1)\diamond G_2(v_2)$ is also written as $G_1(u)\diamond G_2(u)$.
If a connected graph $G$ can be expressed
in the form $G=G_1(u)\diamond G_2(u)$, where $G_1$, $G_2$ are both nontrivial and connected, then $G_1$ is called a
\emph{branch} of $G$ with \emph{root} $u$. Clearly $G_2$ is also a branch of $G$ with root $u$ in the above definition.

We will give some properties of the first  eigenvectors of a connected $k$-uniform $G$ which contains an odd-bipartite branch.
We stress that $k$ is \emph{even} in this and the following sections.

\begin{lemma}\label{prop1}
Let $G=G_0(u)\diamond H(u)$ be a connected $k$-uniform hypergraph, where $H$ is odd-bipartite.
Let $x$ be a first  eigenvector of $G$. Then the following results hold.

\begin{enumerate}
  \item  $x^e \le 0$ for each $e\in E(H)$.
  \item  If $x_u=0$, then $\sum_{e\in E_{G_0}(u)}x^{e\backslash\{u\}}=0$, and $x^{e\backslash\{u\}}=0$ for each $e \in E_H(u)$.
  \item  There exists a first  eigenvector of $G$
  　　　　　　　　　　such that it is nonnegative on one part and nonpositive on the other part for any odd-bipartition of $H$.
\end{enumerate}
\end{lemma}

\begin{proof}
Let $\{U,W\}$ be an odd-bipartition of $H$, where $u\in U$.
Without loss of generality, we assume that $\|x\|_k=1$ and $x_u\ge 0$.
Let $\tilde{x}$ be such that
\[
\tilde{x}_v=
\begin{cases}
x_v,&\text{if } v\in V(G_0)\backslash\{u\};\\
|x_v|,&\text{if } v\in U;\\
-|x_v|,&\text{if } v\in W.
\end{cases}
\]
Note that $\|\tilde{x}\|_k=\|x\|_k=1$, and for each $e \in E(H)$,
\begin{itemize}
    \item [{\rm(a)}] $\tilde{x}_e^{k}=x_e^{k}$.
  \item [{\rm(b)}] $\tilde{x}^e \le x^e$ with equality if and only if $x^e \le 0$.
\end{itemize}

We prove the assertion (1) by a contradiction.
Suppose that there exists an edge $e \in E(H)$ such that $x^{e}>0$.
Then $\tilde{x}^{e}<x^{e}$.
By (a), (b), and Eq. (\ref{form2}), we have
\[
\lambda_{\min}(G)\le  \Q(G)\tilde{x}^k < \Q(G)x^k=\lambda_{\min}(G),
\]
a contradiction.
So $x^e \le  0$ for each $e\in E(H)$, and $\tilde{x}$ is also a first  eigenvector as $\Q(G)\tilde{x}^k = \Q(G)x^k$.
The assertions (1) and (3) follow.

For the assertion (2), let $\bar{x}$ be such that
\[
\overline{x}_v=
\begin{cases}
x_v,&\text{if } v\in V(G_0)\backslash\{u\};\\
-|x_v|,&\text{if } v\in U;\\
|x_v|,&\text{if } v\in W.
\end{cases}
\]
By a similar discussion, $\bar{x}$ is also a first  eigenvector of $G$.
Note that $x_u=0$ and consider the eigenvector equation Eq. (\ref{eigen}) of $\tilde{x}$ and $\bar{x}$ at $u$, respectively.
$$(\lambda_{\min}(G)-d(u))\tilde{x}_u^{k-1}=0=\sum_{e\in E_{G_0}(u)}\tilde{x}^{e\backslash \{u\}}+\sum_{e\in E_H(u)}\tilde{x}^{e\backslash \{u\}},$$
\begin{eqnarray*}
  (\lambda_{\min}(G)-d(u))\bar{x}_u^{k-1}=0 &=& \sum_{e\in E_{G_0}(u)}\bar{x}^{e\backslash \{u\}}+\sum_{e\in E_H(u)}\bar{x}^{e\backslash \{u\}} \\&=& \sum_{e\in E_{G_0}(u)}\tilde{x}^{e\backslash \{u\}}-\sum_{e\in E_H(u)}\tilde{x}^{e\backslash \{u\}}.
\end{eqnarray*}
Thus $\sum_{e\in E_{G_0}(u)}x^{e\backslash\{u\}}=\sum_{e\in E_{G_0}(u)}\tilde{x}^{e\backslash \{u\}}=0$
   and $\sum_{e\in E_H(u)}\tilde{x}^{e\backslash \{u\}}=0$.
As $\tilde{x}^{e\backslash \{u\}} \le 0$ for each edge $e\in E_H(u)$, we have $\tilde{x}^{e\backslash\{u\}}=0$ for each $e \in E_H(u)$.
The assertion (2) follows by the definition of $\tilde{x}$.
\end{proof}

\begin{lemma}\label{prop2}
Let $G=G_0(u)\diamond H(u)$ be a connected non-odd-bipartite $k$-uniform hypergraph, where $H$ is odd-bipartite.
Then $$\lambda_{\min}(G_0)\ge \lambda_{\min}(G),$$
with equality if and only if for any first  eigenvector $y$ of $G_0$, $y_u=0$ and $\tilde{y}$ is a first  eigenvector of $G$,
where $\tilde{y}$ is defined by
\[
\tilde{y}_v=
\begin{cases}
y_v,&\text{if } v\in V(G_0);\\
0,&\text{otherwise.}
\end{cases}
\]
\end{lemma}

\begin{proof}
Suppose that $y$ is a first  eigenvector of $G_0$, $\|y\|_k=1$, and $y_u\ge 0$.
Let $\{U,W\}$ be an odd-bipartition of $H$, where $u \in U$.
Define $\bar{y}$ by
\[
\bar{y}_v=
\begin{cases}
y_v,&\text{if } v\in V(G_0);\\
y_{u},&\text{if } v\in U\backslash\{u\};\\
-y_{u},&\text{if } v\in W.
\end{cases}
\]
Then $\|\bar{y}\|_k=1+(\nu(H)-1)y_u^k$, and
\begin{eqnarray*}
  \Q(G)\bar{y}^k &=& \sum_{e\in E(G)}(\bar{y}_e^{k}+k\bar{y}^e) \\
                        &=& \sum_{e\in E(G_0)}(\bar{y}_e^{k}+k\bar{y}^e)+\sum_{e\in E(H)}(\bar{y}_e^{k}+k\bar{y}^e)\\
            &=& \Q(G_0)y^k +\sum_{e\in E(H)}(ky_u^k-ky_u^k)\\
            &=& \lambda_{\min}(G_0).
\end{eqnarray*}
By Eq. (\ref{form2}), we have
$$\lambda_{\min}(G)\le  \frac{\Q\bar{y}^k} {\|\bar{y}\|_k^k}=\frac{\lambda_{\min}(G_0)} {1+(\nu(H)-1)y_u^k}\le \lambda_{\min}(G_0),$$
where the first equality holds if and only if $\bar{y}$ is also a first  eigenvector of $G$, and the second equality holds if and only if $y_u=0$
(Note that $\lambda_{\min}(G_0)>0$ as $G_0$ is connected and non-odd-bipartite).
The result now follows.
\end{proof}

\begin{coro}\label{prop3}
Let $G=G_0(u)\diamond H(u)$ be a connected non-odd-bipartite $k$-uniform hypergraph, where $H$ is odd-bipartite.
\begin{enumerate}
  \item If $y$ is a first  eigenvector of $G_0$ with $y_u\neq 0$, then
$$\lambda_{\min}(G_0)>\lambda_{\min}(G).$$
  \item If $x$ is a first  eigenvector of $G$ such that $x_u=0$ and $x|_{G_0} \ne 0$, then
$$\lambda_{\min}(G_0)=\lambda_{\min}(G).$$
\end{enumerate}
\end{coro}

\begin{proof}
By Lemma \ref{prop2}, we can get the assertion (1) immediately.
Let $x$ be a first  eigenvector of $G$ as in (2).
By Lemma \ref{prop1}(2), $\sum_{e\in E_{G_0}(u)}x^{e\backslash\{u\}}=0$.
Considering the eigenvector equation (\ref{eigen}) of $x$ at each vertex of $V(G_0)$, we have
$$\Q(G_0)(x|_{G_0})^{k-1}=\lambda_{\min}(G)(x|_{G_0})^{[k-1]}.$$
So $x|_{G_0}$ is an eigenvector of $\Q(G_0)$ associated with the eigenvalue $\lambda_{\min}(G)$.
The result follows by Lemma \ref{prop2}.
\end{proof}

\begin{lemma}\label{bh}
Let $G=G_0(u)\diamond H(u)$ be a connected non-odd-bipartite $k$-uniform hypergraph, where $H$ is odd-bipartite.
If $x$ is a first  eigenvector of $G$, then
\begin{equation}\label{beta}
\beta_H(x):=d_H(u)x_u^k+ \sum_{e\in E_H(u)} x^e \le 0.
\end{equation}
Furthermore, if $\beta_H(x)=0$ and $x|_{G_0} \ne 0$, then $x_u =0$ and $\lambda_{\min}(G_0)=\lambda_{\min}(G)$;
or equivalently if $x_u \ne 0$, then $\beta_H(x)<0$.
\end{lemma}

\begin{proof}
Let $\la:=\la_{\min}(G)$.
By Eq. (\ref{eigen}),
for each $v\in V(G_0)\backslash \{u\}$,
\begin{equation}\label{sem}
((\Q(G)-\lambda \I) x^{k-1})_v= ((\Q(G_0)-\lambda\I)(x|_{G_0})^{k-1})_v=0.
\end{equation}
For the vertex $u$,
\begin{eqnarray*}
\lambda x_u^{k-1} =(\Q(G)x^{k-1})_u &= & d_G(u)x_u^{k-1}+ \sum_{e\in E_G(u)}x^{e\backslash \{u\}},\\
& = & d_{G_0}(u)x_u^{k-1}+ \sum_{e\in E_{G_0}(u)}x^{e\backslash \{u\}}+d_H(u)x_u^{k-1}+ \sum_{e\in E_H(u)} x^{e\backslash \{u\}}\\
&= & (\Q(G_0)(x|_{G_0})^{k-1})_u +d_H(u)x_u^{k-1}+ \sum_{e\in E_H(u)} x^{e\backslash \{u\}}.
\end{eqnarray*}
So,
\begin{equation}\label{u}
((\Q(G_0)-\lambda\I)(x|_{G_0})^{k-1})_u=-\left(d_H(u)x_u^{k-1}+ \sum_{e\in E_H(u)} x^{e\backslash \{u\}}\right).
\end{equation}

By Lemma \ref{prop2} and Lemma \ref{semi}(1), $\Q(G_0)-\lambda\I$ is positive semidefinite.
Then $(\Q(G_0)-\lambda\I)y^k \ge 0$ for any real and nonzero $y$.
So, by Eq. (\ref{sem}) and Eq. (\ref{u}), we have
\begin{eqnarray*}
0 \le (\Q(G_0)-\lambda\I)(x|_{G_0})^k &=& (x|_{G_0})^\top ((\Q(G_0)-\lambda\I)(x|_{G_0})^{k-1} \\
                                        &=& -x_u \left(d_H(u)x_u^{k-1}+ \sum_{e\in E_H(u)} x^{e\backslash \{u\}}\right) \\
                                         &=& -\beta_H(x).
\end{eqnarray*}
So we have $\beta_H(x)\le 0$.

Suppose that $\beta_H=0$ and $x|_{G_0} \ne 0$.
If $x_u=0$, by Corollary \ref{prop3}(2), $\lambda_{\min}(G_0)=\lambda_{\min}(G)$.
If $x_u\neq 0$, then $d_H(u)x_u^{k-1}+ \sum_{e\in E_H(u)} x^{e\backslash \{u\}}=0$.
By Eq. (\ref{sem}) and Eq. (\ref{u}), $(\lambda_{\min}(G),x|_{G_0})$ is an eigenpair of $\Q(G_0)$, implying that $\lambda_{\min}(G)=\lambda_{\min}(G_0)$ by Lemma  \ref{prop2}.
However, $(x|_{G_0})_u =x_u \ne 0$, a contradiction to Corollary \ref{prop3}(1).
\end{proof}

\begin{lemma}\label{nonzero}
Let $G=G_0(r)\diamond T(r)$ be a connected non-odd-bipartite $k$-uniform hypergraph, where $T$ is a power hypertree.
If $x$ is a first  eigenvector of $G$ and $x_p\neq 0$ for some $p\in V(T)$,  then $x_q\neq0$
whenever $q$ is a vertex of $T$ such that $p$ lies on the unique path from $r$ to $q$.
\end{lemma}

\begin{proof}
It suffices to consider three vertices $u, v, w$ in a common edge $e\in E(T)$, where $d(u)\ge 2$, $d(v)=1$, $d(w)\ge 2$, and $u$ lies on the path from $r$ to $w$.
We will show $x_u \ne 0 \Rightarrow x_v \ne 0 \Rightarrow x_w \ne 0$.
Write $G=\bar{G}_0(u) \diamond \bar{T}(u)$,
where $\bar{G}_0$ contains $G_0$ as a sub-hypergraph, and $\bar{T}$ is a sub-hypergraph of $T$ such that $e$ is the only edge of $\bar{T}$ containing $u$.
Suppose that $x_u\neq 0$.
If $x_v=0$, by Lemma \ref{bh}, $$\beta_{\bar{T}}(x)=x_u^k+x^e=x_u^k \le 0,$$ a contradiction.
So $x_v \ne 0$.
If $x_w=0$, then by Eq. (\ref{eigen}), $x_v=0$ as $\la_{\min}(G)< \delta(G)=1$ by Lemma \ref{mindegree}, also a contradiction.
So $x_w \ne 0$.
\end{proof}

\begin{lemma}\label{increase}
Let $G=G_0(r)\diamond T(r)$ be a connected non-odd-bipartite $k$-uniform hypergraph, where $T$ is a power hypertree.
If $x$ is a first  eigenvector of $G$ and $x_r\neq 0$, then $|x_u|< |x_w|$ whenever $u,w$ are two vertices of $T$ such that
$u$ lies on the unique path from $r$ to $w$, $d(u)\ge 2$ and $d(w)\ge 2$.
\end{lemma}

\begin{proof}
By Lemma \ref{nonzero}, $x_i\neq0$ for any vertex $i\in V(T)$.
It suffices to consider three vertices $u, v, w$ in a common edge $e \in E(T)$, where $d(u)\ge 2$, $d(v)=1$, $d(w)\ge 2$, and $u$ lies on the path from $r$ to $w$.
We will show that $|x_u|< |x_w|$.
By the eigenvector equation of $x$ at $v$, noting that $0<\la:=\la_{\min}(G)<1$, by Lemma \ref{xuv} we have
$$(1-\lambda)|x_v|^k=|x^e|=|x_u\|x_v|^{k-2}|x_w|,$$
which implies that
\begin{equation}\label{xv}
|x_v|=\left(\frac{|x_u\|x_w|} {1-\lambda}\right)^{\frac{1}{2}}.
\end{equation}

We can write $G=\bar{G}_0(w) \diamond \bar{T}(w)$, where $\bar{G}_0$ contains $G_0$ and the edge $e$ as the only one containing $w$.
Then by Lemma \ref{bh}, noting $x_w \ne 0$,
$$\beta_{\bar{T}}(x)=(d_T(w)-1)x_w^k+ \sum_{\tilde{e} \in E_T(w)\backslash\{e\}}x^{\tilde{e}} < 0.$$
By Lemma \ref{prop1}(1), we have
\begin{equation}\label{xw}
(d_T(w)-1)|x_w|^k <  \sum\limits_{\tilde{e}\in E_T(w)\backslash\{e\}}|x^{\tilde{e}}|.
\end{equation}

Considering the eigenvector equation of $x$ at $w$, by Lemma \ref{xuv}, Eq. (\ref{xv}) and Eq. (\ref{xw}), we have
\begin{eqnarray*}
  (d_T(w)-\lambda)|x_w|^k &=& \sum\limits_{\tilde{e}\in E_T(w)\backslash\{e\}}|x^{\tilde{e}}|+|x_u\|x_v|^{k-2}|x_w| \\
    &> & (d_T(w)-1)|x_w|^k+|x_u| \left(\frac{|x_u\|x_w|} {1-\lambda}\right)^{\frac{k}{2}-1}|x_w|\\
    &=& (d_T(w)-1)|x_w|^k+({1-\lambda})^{1-\frac{k}{2}}|x_u|^{\frac{k}{2}}|x_w|^{\frac{k}{2}}.
\end{eqnarray*}
So
$$({1-\lambda})|x_w|^k>  ({1-\lambda})^{1-\frac{k}{2}}|x_u|^{\frac{k}{2}}|x_w|^{\frac{k}{2}},$$
and hence
$$ |x_w|^{\frac{k}{2}}> ({1-\lambda})^{\frac{k}{2}} |x_w|^{\frac{k}{2}}>|x_u|^{\frac{k}{2}}.$$
\end{proof}

%
%
%

Denote by $G=G_0(r)\diamond P_m^k(r)$ the coalescence of $G_0$ and $P_m^k$ by identifying one vertex of $G_0$ and one pendent vertex of $P_m^k$ and forming a new vertex $r$.

\begin{lemma}\label{formula}
Let $G=G_0(r)\diamond P_m^k(r)$ be a connected non-odd-bipartite $k$-uniform hypergraph, where $P_m^k$ is a hyperpath with $m$ edges.
Starting from the root $r$, label edges of $P_m^k$ as $e_m, e_{m-1},\ldots,e_1$, and some vertices of those edges as
\begin{equation}\label{lab}
r=2m, 2m-1, 2m-2, \ldots, 2i, 2i-1, 2i-2, \ldots, 2,1,0,
\end{equation}
where  $\{2i, 2i-1, 2i-2)\} \subseteq e_i$ for $i \in [m]$, $d_{P_m}(2i)=2$ and $d_{P_m}(2i-1)=1$ for $i \in [m-1]$, $d_{P_m}(2m)=d_{P_m}(0)=1$.
If $x$ is a first  eigenvector of $G$ and $x_r\neq 0$,
Then
\begin{equation}\label{recur}|x_{2i}|={f_i(\la_{\min}(G))}^{\frac{2}{k}}|x_0|,\ i\in[m],
\end{equation}
where $f_i(x)$ is defined recursively as $f_0(x)=1$, $f_1(x)=(1-x)^{\frac{k}{2}}$,
$$f_{i+1}(x)=(2-x)(1-x)^{\frac{k}{2}-1}f_i(x)-f_{i-1}(x),\ i\in[m-1].$$
Furthermore,
$0<f_i(\la_{\min}(G))<1$, and $f_i(\la_{\min}(G))$ is strictly decreasing in $i$.
\end{lemma}

\begin{proof}
By Lemma \ref{nonzero}, $x_v \ne 0$ for each $v \in V(P_{m})$.
 Let $\la:=\la_{\min}(G)$.
 Then $0<\la <1$ by Lemma \ref{mindegree}. 
By Lemma \ref{xuv} and Eq. (\ref{eigen}), $$|x_0|= |x_1|, \;|x_2|=(1-\la)|x_0|.$$
So, by Lemma \ref{xuv}, Lemma \ref{prop1}(1) and Eq. (\ref{xv}), considering the eigenvector equation of $x$ at the vertex $2i$ ($1 \le i \le m-1$), we have
\begin{eqnarray*}
   (2-\lambda)|x_{2i}|^{k-1} &=& |{x_{2i-1}}|^{k-2}|x_{2i-2}|+|{x_{2i+1}}|^{k-2}|x_{2i+2}| \\
   &=& \left(\frac{|x_{2i}x_{2i-2}|} {1-\lambda}\right)^{\frac{k}{2}-1}|x_{2i-2}|+\left(\frac{|x_{2i}x_{2i+2}|} {1-\lambda}\right)^{\frac{k}{2}-1}|x_{2i+2}|\\
   &=& ({1-\lambda})^{1-\frac{k}{2}} |x_{2i}|^{\frac{k}{2}-1}\left(|x_{2i-2}|^{\frac{k}{2}}+|x_{2i+2}|^{\frac{k}{2}}\right).
\end{eqnarray*}
Thus, for $i=1,\ldots,m-1$,
$$ |x_{2i+2}|^{\frac{k}{2}}=(2-\lambda)({1-\lambda})^{\frac{k}{2}-1}|x_{2i}|^{\frac{k}{2}}-|x_{2i-2}|^{\frac{k}{2}}.$$
It is easy to verify that for $i \in [m]$,
$$|x_{2i}|^{\frac{k}{2}}={f_i(\la)}|x_0|^{\frac{k}{2}}.$$
By Lemma \ref{increase}, for $i\in[m]$,
$0<f_i(\la)<1$, and $f_i(\la)$ is strictly decreasing in $i$.
\end{proof}

\section{Perturbation of the least  eigenvalues}

We first give a perturbation result on the least  eigenvalues under relocating an odd-bipartite branch.
Let $G_0$, $H$ be two vertex-disjoint hypergraphs, where $v_1$, $v_2$ are two distinct vertices of $G_0$, and $u$ is a vertex of $H$ (called the \emph{root} of $H$).
Let $G = G_1(v_2)\diamond H(u)$ and $\tilde{G} =G_1(v_1)\diamond H(u)$.
We say that $\tilde{G}$ is obtained from $G$ by \emph{relocating} $H$ rooted at $u$ from $v_2$ to $v_1$; see Fig. \ref{relo}.

\begin{figure}[htbp]
\includegraphics[scale=.8]{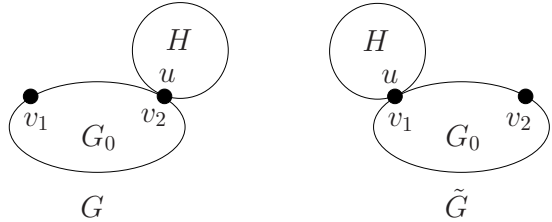}
\caption{\small Relocating $H$ from $v_2$ to $v_1$}\label{relo}
\end{figure}

\begin{lemma}\label{movedge}
Let $G=G_0(v_2)\diamond H(u)$ and $\tilde{G}=G_0(v_1)\diamond H(u)$ be connected non-odd-bipartite $k$-uniform hypergraphs, where $H$ is odd-bipartite.
If $x$ is a first  eigenvector of $G$ such that $|x_{v_1}|\ge |x_{v_2}|$,  then
$$\lambda_{\min}(\tilde{G})\le \lambda_{\min}(G),$$
with equality if and only if $x_{v_1}=x_{v_2}=0$, and $\tilde{x}$ defined in (\ref{case2}) is a first eigenvector of $\tilde{G}$.
\end{lemma}

\begin{proof}
Let $x$ be a first  eigenvector of $G$ such that $\|x\|_k=1$ and $x_{v_1}\ge 0$.
We divide the discussion into three cases.
Denote $\la:=\lambda_{\min}(G)$.

\textbf{Case 1:} $x_{v_2}>0$.
Write $x_{v_1}=\delta x_{v_2}$, where $\delta \ge  1$.
Define $\tilde{x}$ on $\tilde{G}$ by
\begin{equation}\label{evtld1}
\tilde{x}_v=
\begin{cases}
x_v,&\text{if } v\in V(G_0);\\
\delta x_v,&\text{if } v\in V(H)\backslash\{u\}.
\end{cases}
\end{equation}
Then $\|\tilde{x}\|_k^k=1+(\delta^k-1)\sum_{v\in {V(H)\backslash\{u\}}}x_v^k$, and
\begin{eqnarray*}
  \Q(\tilde{G})\tilde{x}^k &=& \sum_{e\in E(\tilde{G})}(\tilde{x}_e^{k}+k \tilde{x}^e) \\
            &=& \Q(G)x^k +(\delta^k-1)\sum_{e\in E(H)}(x_e^{k}+kx^e)\\
            &=& \la+(\delta^k-1)\sum_{e\in E(H)}(x_e^{k}+kx^e).
\end{eqnarray*}

By the eigenvector equation of $x$ at each vertex $v\in V(H)\backslash \{u\}$,
\begin{equation}\label{ehv}
d(v)x_v^k+\sum_{e\in E_H(v)}x^e=\la x_v^k.
\end{equation}
By the eigenvector equation of $x$ at $u$,
\begin{equation}\label{egu}
d(u)x_u^k+\sum_{e\in E_G(u)}x^e=\alpha_{G_0}(x)+\beta_H(x)=\la x_u^k,
\end{equation}
where $\alpha_{G_0}(x):=d_{G_0}(u)x_u^k+ \sum_{e\in E_{G_0}(u)}x^e$.
By Eq. (\ref{ehv}) and Eq. (\ref{egu}), we have
$$\alpha_{G_0}+\sum_{e\in E(H)}(x_e^{k}+kx^e)=\la \sum_{v\in V(H)}x_v^k.$$
So
\begin{eqnarray*}
  \sum_{e\in E(H)}(x_e^{k}+kx^e) &=& \la \sum_{v\in V(H)}x_v^k-\alpha_{G_0}(x) \\
  &=& \la \sum _{v\in V(H)}x_v^k-(\la x_u^k-\beta_H(x))\\
  &=& \la \sum_{v\in {V(H)\backslash\{u\}}}x_v^k+\beta_H(x).
\end{eqnarray*}
Thus
\begin{eqnarray*}
   \Q(\tilde{G})\tilde{x}^k &=&  \la+(\delta^k-1)\left(\la \sum_{v\in {V(H)\backslash\{u\}}}x_v^k+\beta_H(x)\right)\\
            &=& \la \left(1+(\delta^k-1)\sum_{v\in {V(H)\backslash\{u\}}}x_v^k\right)+(\delta^k-1)\beta_H(x)\\
            &=& \la \|\tilde{x}\|_k^k+(\delta^k-1)\beta_H(x).
\end{eqnarray*}
As $x_{v_2} \ne 0$, $\beta_H(x) < 0$ by Lemma \ref{bh},
\[
\lambda_{\min}(\tilde{G})\le  \frac{\Q(\tilde{G})\tilde{x}^k} {\|\tilde{x}\|_k^k}
=\lambda+\frac{(\delta^k-1)\beta_H(x)} {\|\tilde{x}\|_k^k}\le \la=\lambda_{\min}(G),
\]
where the first equality holds if and only if $\tilde{x}$ is a first  eigenvector of $\tilde{G}$,
and the second equality holds if and only if $\delta=1$, i.e., $x_{v_1}=x_{v_2}$.
By the eigenvector equations of $x$ and $\tilde{x}$ at $v_2$ respectively, we will get $\beta_H(x)=0$, a contradiction.
So, in this case, $\lambda_{\min}(\tilde{G})< \lambda_{\min}(G)$.

\textbf{Case 2:} $x_{v_2}=0$.
First assume $x_{v_1}=0$.
Define $\tilde{x}$ on $\tilde{G}$ by
\begin{equation}\label{case2}
\tilde{x}_v=
\begin{cases}
x_v,&\text{if } v\in V(G_0);\\
x_v,&\text{if } v\in V(H)\backslash\{u\}.
\end{cases}
\end{equation}
Then $\|\tilde{x}\|_k^k=1$,
and
$$\la_{\min}(\tilde{G}) \le \Q(\tilde{G})\tilde{x}^k=\Q(G)x^k =\la_{\min}(G),$$
with equality if and only if $\tilde{x}$ is a first  eigenvector of $\tilde{G}$.

Now assume that $x_{v_1}>0$.
By Corollary \ref{prop3}(2) and its proof, $\la_{\min}(G)=\la_{\min}(G_0)$ as $x_{u}=0$ and $x|_{G_0} \ne 0$;
furthermore, $x|_{G_0}$ is a first  eigenvector of $G_0$.
By  Corollary \ref{prop3}(1), $\la_{\min}(G_0)>\la_{\min}(\tilde{G})$ as $(x|_{G_0})_{v_1}\ne 0$,
thinking of $v_1$ a coalescence vertex between $G_0$ and $H$ in $\tilde{G}$.
So $\la_{\min}(\tilde{G}) < \la_{\min}(G)$.

\textbf{Case 3:} $x_{v_2}<0$.
Write $x_{v_1}=-\delta x_{v_2}$, where $\delta \ge  1$.
Define $\tilde{x}$ on $\tilde{G}$ by
\begin{equation}\label{evtld2}
\tilde{x}_v=
\begin{cases}
x_v,&\text{if } v\in V(G_0);\\
-\delta x_v,&\text{if } v\in V(H)\backslash\{u\}.
\end{cases}
\end{equation}
By a similar discussion to Case 1 by replacing $\delta$ by $-\delta$,  we also have
$\lambda_{\min}(\tilde{G})< \lambda_{\min}(G)$.
\end{proof}

%
%
%

\begin{coro}\label{gst}
Let $G$ be a connected non-odd-bipartite $k$-uniform hypergraph,
 and  $G_{s,t}$ be the hypergraph obtained by coalescing $G$ with two hyperpaths $P_s^k$ and $P_t^k$
 by identifying a pendent vertex of $P_s^k$ and a pendent vertex of $P_t^k$ both with a vertex $u$ of $G$, where $s \ge t \ge 1$.
 If $x$ is a first Q-eigenvector of $G_{s,t}$ and $x_u\neq 0$, then
$$\lambda_{\min}(G_{s,t})>\lambda_{\min}(G_{s+1,t-1}).$$
\end{coro}

\begin{proof}
Using the method of labeling vertices as in Eq. (\ref{lab}), we label some of the vertices $P_s^k$ as
$$u=2s, 2s-1, 2s-2, \ldots, 2,1,0,$$
and label some of the vertices of $P_t^k$ as
$$u =2\bar{t}, 2\bar{t}-1, 2\bar{t}-2, \ldots,\bar{2},\bar{1},\bar{0}.$$
Then, by Lemma \ref{formula}
$$|x_u|=f_s(\la)^{2 \over k}|x_0|=f_t(\la)^{2 \over k}|x_{\bar{0}}|,$$
where $\la:=\lambda_{\min}(G_{s,t})$, and $f_i(x)$ is defined as in Lemma \ref{formula}.
As $s \ge t$, $0< f_s(\la) \le f_t(\la)$ also by Lemma \ref{formula}.
So, combining the eigenvector equation on $\bar{0}$, $|x_0| \ge |x_{\bar{0}}| > |x_{\bar{2}}|>0$.
Now relocating the pendent edge of $P_t^k$ rooted at $\bar{2}$ and attaching to the pendent vertex $0$ of $P_s^k$, we arrive at the hypergraph $G_{s+1,t-1}$.
The result follows by Lemma \ref{movedge}.
\end{proof}

\begin{lemma}\label{splust}
Let $G$ and $G_{s,t}$ be as defined in Corollary \ref{gst}.
Then
$$\lambda_{\min}(G_{s,t}) \ge \lambda_{\min}(G_{s+t,0}).$$
Furthermore, if $x$ is a first Q-eigenvector of $G_{s,t}$ and $x_u \neq 0$,
then
$$\lambda_{\min}(G_{s,t})>\lambda_{\min}(G_{s+t,0}).$$
\end{lemma}

\begin{proof}
Suppose that the labeling of some vertices of $P_s^k$ and $P_t^k$ is as in the proof of Corollary \ref{gst}.
Let $x$ be a first  eigenvector of $G_{s,t}$.
If $x_{u} = 0$, then $|x_{\bar{0}}| \ge |x_{u}|$; otherwise $|x_{\bar{0}}| > |x_{u}|$ by Lemma \ref{increase}.
Relocating $P_t^k$ rooted at $u$ and attaching to the pendent vertex $0$ of $P_s^k$, we arrive at the hypergraph $G_{s+t,0}$.
The result follows by Lemma \ref{movedge}.
\end{proof}

A hypergraph is called a \emph{minimizing hypergraph} in a certain class of hypergraphs if its least  eigenvalue attains the minimum
among all hypergraphs in the class.
Denoted by $\mathcal{T}_m(G_0)$ the class of hypergraphs with each obtained from a fixed connected non-odd-bipartite hypergraph $G_0$ by attaching
some hypertrees at some vertices of $G_0$ respectively
(i.e. identifying a vertex of a hypertree with some vertex of $G_0$ each time) such that the number of its edges equals $\varepsilon(G_0)+m$.
We will characterize the minimizing hypergraph(s) in $\mathcal{T}_m(G_0)$.

\begin{theorem}\label{min}
Let $G_0$ be a connected non-odd-bipartite $k$-uniform hypergraph.
If $G$ is a minimizing hypergraph in $\mathcal{T}_m(G_0)$, then $G=G_0(u) \diamond P_m(u)$ for some vertex $u$ of $G_0$.
\end{theorem}

\begin{proof}
Suppose that $G$ is a minimizing hypergraph in $\mathcal{T}_m(G_0)$, and $G$ has no the structure as desired in the theorem.
We will get a contradiction by the following three cases.

\textbf{Case 1:} $G$ contains hypertrees attached at two or more vertices of $G_0$.
Let $T_1$, $T_2$ be two hypertrees attached at $v_1,v_2$ of $G_0$ respectively.
Let $x$ be a first  eigenvector of $G$.
Assume $|x_{v_1}| \ge |x_{v_2}|$.
Relocating $T_2$ rooted at $v_2$ and attaching to $v_1$, we will get a hypergraph $\bar{G} \in \mathcal{T}_m(G_0)$ such
that $\la_{\min}(\bar{G}) \le \la_{\min}(G)$ by Lemma \ref{movedge}.
Repeating the above operation, we finally arrive at a hypergraph $G^{(1)}$ with only one hypertree $T^{(1)}$ attached at one vertex $u_0$ of $G_0$ such that
$\la_{\min}(G^{(1)}) \le \la_{\min}(G)$.

{\bf Case 2:} $T^{(1)}$ contains edges with three or more vertices of degree greater than one, i.e. $T^{(1)}$ is not a power hypertree.
Let $e$ be one of such edges containing $u,v,w$ with $d(u),d(v), d(w)$ all greater than one.
Let $x$ be a first  eigenvector of $G^{(1)}$, and assume that $|x_u| \ge |x_w|$.
Relocating the hypertree rooted at $w$ and attaching to $u$, we will get a hypergraph $\hat{G} \in \mathcal{T}_m(G_0)$ such
that $\la_{\min}(\hat{G}) \le \la_{\min}(G^{(1)})$ by Lemma \ref{movedge}.
Repeating the above operation on the edge $e$ until $e$ contains exactly $2$ vertices of degree greater than one, and on each other edges like $e$,
we finally arrive at a hypergraph $G^{(2)}$ such that the unique hypertree $T^{(2)}$ attached at $u_0$ is a power hypertree, and
$\la_{\min}(G^{(2)}) \le \la_{\min}(G^{(1)})$.

{\bf Case 3:} $T^{(2)}$ contains more than one pendent edges except the edge(s) containing $u_0$.
Let $x$ be a first  eigenvector of $G^{(2)}$.
We assert that $|x_{u_0}|=\max_{v \in V(G_0)}|x_v|$.
Otherwise, there exists a vertex $v_0$ of $G$ such that $|x_{v_0}|>|x_{u_0}|$.
Relocating $T^{(2)}$ rooted at $u_0$ and attaching to $v_0$, we will get a hypergraph $\tilde{G} \in \mathcal{T}_m(G_0)$ such
that $\la_{\min}(\tilde{G}) < \la_{\min}(G^{(2)})$ by Lemma \ref{movedge}.
Then $\la_{\min}(\tilde{G}) < \la_{\min}(G)$, a contradiction to $G$ being minimizing.
We also assert that there exists one pendent vertex $w_0$ of $T^{(2)}$ such that $x_{w_0} \ne 0$.
Otherwise by Lemma \ref{nonzero}, $x|_{T^{(2)}}=0$, in particular $x_{u_0}=0$, and hence $x=0$ by the first assertion, a contradiction.

Note that $T^{(2)}$ consists of $d_{T^{(2)}}(u_0)$ sub-hypertrees sharing a common vertex $u_0$.
Let $T^{(2)}_1$ be the sub-hypertrees of $T^{(2)}$ attached at $u_0$ which contains $w_0$.
If  $d_{T^{(2)}}(u_0)=1$, let $p$ be the furthest vertex of degree greater $2$ on the path starting from $u_0$ to $w_0$,
and let $T_p$ be the hypertree attached to $p$ which contains no vertices of the path except $p$.
Relocating $T_p$ rooted at $p$ and attaching to $w_0$,
we will arrive at a hypergraph still in $\mathcal{T}_m(G_0)$ but with a smaller least  eigenvalue by Lemma \ref{increase} and Lemma \ref{movedge}
regardless of $x_p$ being zero or not, a contradiction.
If  $d_{T^{(2)}}(u_0)>1$, let $T^{(2)}_2$ be the sub-hypertree of $T^{(2)}$  attached at $u_0$ which contains no $w_0$.
Relocating $T^{(2)}_2$ from $u_0$ and attaching to $w_0$,
we still arrive at a hypergraph in $\mathcal{T}_m(G_0)$ but with a smaller least  eigenvalue, also a contradiction.
The result now follows.
\end{proof}

\section{Least limit point of the least eigenvalues}
In this section we will investigate the upper bounds of the least  eigenvalues, from which we show that the
least limit point of the least  eigenvalues of connected non-odd-bipartite hypergraphs is zero.

\begin{lemma}\label{odd}
Let $G$ be a non-odd-bipartite $k$-uniform hypergraph.
Then $G$ contains an odd-bipartite sub-hypergraph with at least $\frac{\varepsilon(G)}{2}$ edges.
\end{lemma}

\begin{proof}
Let $T \subseteq V(G)$ be a random subset given by $\Pr[v\in T]=\frac{1}{2}$, these choices being mutually independent.
Set $B=V(G)\backslash T$.
Call an edge $e$ odd-transversal if exactly the cardinality of $e\cap T$ is odd. Let $X$ be the number of odd-transversal edges.
We decompose
$$X =\sum_{e\in E(G)}X_e,$$
where $X_e$ is the indicator random variable for $e$ being odd-transversal, i.e,
$X_e=1$ if $e$ is odd-transversal, and $X_e=0$ otherwise.
Then the expectation
\[
E[X_e]=\Pr[X_e] = \sum_{i \text{\footnotesize~is odd}, i\in[k]} {k \choose i} \left(\frac{1}{2}\right)^i\left(\frac{1}{2}\right)^{k-i}=\frac{1}{2}.
\]
So
$E(X) =\sum_{e\in E(G)}E(X_e)=\frac{\varepsilon(G)}{2}$.
Thus $X \ge \frac{\varepsilon(G)}{2}$ for some choice of $T$, and the set of those odd-transversal edges forms an odd-bipartite sub-hypergraph.
\end{proof}

\begin{theorem}\label{up1}
Let $G=G_0(u)\diamond H(u)$ be a connected non-odd-bipartite $k$-uniform hypergraph, where $H$ is odd-bipartite.
Then
$$\lambda_{\min}(G)\le  \frac{k\varepsilon(G_0)}{\nu(G)}.$$
\end{theorem}

\begin{proof}
By Lemma \ref{odd}, there is a proper subset $T$ of $V(G_0)$ such that the number of odd transversal edges of $G_0$ respect to $T$ is at least $\frac{\varepsilon(G_0)}{2}$.
Let $\{U,W\}$ be an odd-bipartition of $H$, where $u\in U$.
Define $x$ by
\[
x_v=
\begin{cases}
1,&\text{if } v\in T\cup U;\\
-1,&\text{otherwise}.
\end{cases}
\]
Then $\|x\|_k^k=\nu(G)$.
We write $e\sim T$ (or $e\nsim T$) to denote that $e$ is odd-transversal (or not odd-transversal) respect to $T$.
By Eq. (\ref{form2}) and Lemma \ref{odd},
\begin{eqnarray*}
\lambda_{\min}(G)&\le & \frac{\Q x^k} {\|x\|_k^k}\\
            &=& \frac{1}{\nu(G)} \left(\sum_{e\in E(G_0), e\nsim T}(x_e^{k}+kx^e)+\sum_{e\in E(G_0), e\sim T}(x_e^{k}+kx^e)
            +\sum_{e\in E(H)}(x_e^{k}+kx^e)\right)\\
            &= & \frac{1}{\nu(G)} \sum_{e\in E(G_0), e\nsim T}2k\\
              &\le& \frac{k\varepsilon(G_0)}{\nu(G)}.
\end{eqnarray*}
\end{proof}

\begin{theorem}\label{up2}
Let $G=G_0(u)\diamond H(u)$ be a connected non-odd-bipartite $k$-uniform hypergraph, where $H$ is odd-bipartite.
 Then
$$\lambda_{\min}(G)\le  \frac{d_{G_0}(u)}{\nu(H)}.$$
\end{theorem}

\begin{proof}
Let $\{U,W\}$ be an odd-bipartition of $H$, where $u\in U$. Define $x$ by
\[
x_v=
\begin{cases}
1,&\text{if } v\in U;\\
-1,&\text{if } v\in W;\\
0,&\text{otherwise}.
\end{cases}
\]
Then $\|x\|_k^k=\nu(H)$, and
\begin{eqnarray*}
\lambda_{\min}(G)&\le & \frac{\Q x^k} {\|x\|_k^k}\\
            &=& \frac{1}{\nu(H)}\left(\sum_{e\in E(G_0)}(x_e^{k}+kx^e)+\sum_{e\in E(H)}(x_e^{k}+kx^e)\right)\\
            &=& \frac{d_{G_0}(u)}{\nu(H)}.
\end{eqnarray*}
\end{proof}

\begin{remark}\label{rmk2}
In Theorem \ref{up1}, the upper bound
$$\frac{k\varepsilon(G_0)}{\nu(G)}=\frac{k\varepsilon(G_0)}{\nu(G_0)}\frac{\nu(G_0)}{\nu(G)}=d(G_0)\frac{\nu(G_0)}{\nu(G)},$$
where $d(G_0)$ is the average degree of the vertices of $G_0$.
So, if fixing $G_0$, and letting $H$ have enough vertices, then the bounds in Theorems \ref{up1} and \ref{up2} will be much smaller than $\delta(G)$, the upper bound in Lemma \ref{mindegree}.
\end{remark}

By Theorem \ref{up2} and Lemma \ref{tree}, we get the following result immediately.

\begin{coro}\label{up3}
Let $G=G_0(u)\diamond T_m(u)$ be a connected non-odd-bipartite $k$-uniform hypergraph, where $T_m(u)$ is a hypertree with $m$ edges. Then
$$\lambda_{\min}(G)\le  \frac{d_{G_0}(u)}{(k-1)m+1}.$$
\end{coro}

By Lemma \ref{prop2}, $\lambda_{\min}(G_0(r)\diamond P_m^k(r))$ and $\lambda_{\min}(G_0(r)\diamond S_m^k(r))$  are both decreasing in $m$,
which implies that they have limits.
By Corollary \ref{up3}, those two limits are both $0$.
As for  a connected non-odd-bipartite hypergraph, its least  eigenvalue is greater than $0$.
So we get the following result.

\begin{coro}\label{last}
Zero is the least limit point of the least  eigenvalues of connected non-odd-bipartite hypergraphs.
\end{coro}

\bibliographystyle{amsplain}

\end{document}